\documentclass[11pt,reqno,a4paper]{amsart}

\usepackage{mathtools}
\usepackage{etoolbox}
\usepackage{amsmath}
\usepackage{enumerate}
\usepackage{amssymb}
\usepackage{amscd} 
\usepackage{amsthm}
\usepackage{amsfonts}
\usepackage{cancel}
\usepackage[usenames,dvipsnames,svgnames]{xcolor}
\usepackage{tikz}
\usepackage{graphicx}
\usepackage[matrix, arrow, curve]{xy}
\usepackage{comment}
\usepackage[scaled]{helvet} 
\usepackage{courier} 
\usepackage[T1]{fontenc}                                                    
\usepackage{hyperref}
\usepackage{imakeidx}
\usepackage{extarrows}
\usepackage[utf8]{inputenc}
\usepackage{amsthm}

\numberwithin{equation}{section}

\newcommand{\im}{\operatorname{Im}}

\newcommand{\cT}{\mathcal{T}}
\newcommand{\cU}{\mathcal{U}}

\newcommand{\N}{\mathbb{N}}

\newcommand{\Z}{\mathbb{Z}}

\newcommand\subsetsim{\mathrel{%
\ooalign{\raise0.2ex\hbox{$\subset$}\cr\hidewidth\raise-0.8ex\hbox{\scalebox{0.9}{$\sim$}}\hidewidth\cr}}}

\newcommand{\asdim}{\operatorname{asdim}}

\newcommand{\diam}{\operatorname{diam}}

\renewcommand{\phi}{\varphi}

\DeclareMathOperator{\proj}{proj}

\theoremstyle{thm}
\newtheorem{theorem}{Theorem}[section]
\newtheorem{corollary}[theorem]{Corollary}
\newtheorem{proposition}[theorem]{Proposition}
\newtheorem{lemma}[theorem]{Lemma}

\newenvironment{usethmcounterof}[1]{%
  \renewcommand\thetheorem{\ref{#1}}\theorem}{\endtheorem\addtocounter{theorem}{-1}}

\theoremstyle{definition}
\newtheorem{definition}[theorem]{Definition}
\newtheorem{notation}[theorem]{Notation}

\newtheorem{remark}[theorem]{Remark}
\newtheorem{example}[theorem]{Example}

\makeindex

\patchcmd{\section}{-.5em}{.5em}{}{}
\patchcmd{\subsubsection}{-.5em}{.5em}{}{}

\makeatletter
\@namedef{subjclassname@2020}{%
  \textup{2020} Mathematics Subject Classification}
\makeatother

\makeatletter
\let\@wraptoccontribs\wraptoccontribs
\makeatother

\begin{document}

\title[Hurewicz and Dranishnikov-Smith for ASDIM of countable approximate groups]{Hurewicz and Dranishnikov-Smith theorems for asymptotic dimension of countable approximate groups}

\author{Tobias Hartnick}
\address{Institut f\"ur Algebra und Geometrie, KIT, Germany}
\curraddr{}
\email{tobias.hartnick@kit.edu}
\thanks{}

\author{Vera Toni\'c}
\address{Faculty of Mathematics, University of Rijeka, Croatia}
\curraddr{}
\email{vera.tonic@math.uniri.hr}
\thanks{}



\date{\today}

\begin{abstract} We establish two main results for the asymptotic dimension of countable approximate groups. The first one is a Hurewicz type formula for a global morphism of countable approximate groups $f:(\Xi, \Xi^\infty) \to (\Lambda, \Lambda^\infty)$, stating that $\asdim \Xi \leq \asdim \Lambda +\asdim ([\ker f]_c)$. This is analogous to the Dranishnikov-Smith result for groups, and is relying on another Hurewicz type formula we prove, using a $6$-local morphism instead of a global one. The second result is similar to the Dranishnikov-Smith theorem stating that, for a countable group $G$, $\asdim G$ is equal to the supremum of asymptotic dimensions of finitely generated subgroups of $G$. Our version states that, if $(\Lambda, \Lambda^\infty)$ is a countable approximate group, then $\asdim \Lambda$ is equal to the supremum of asymptotic dimensions of approximate subgroups of finitely generated subgroups of $\Lambda^\infty$, with these approximate subgroups contained in $\Lambda^2$.
\end{abstract}

\maketitle

\section{Introduction -- Summary of results}
The basic idea behind introducing approximate groups can be stated as -- loosening, in a controlled way, the rule that a subgroup should be closed under the group operation. Relaxing the strict rules in algebra is not new; for example,
in his book \cite{Ulam}, S.~Ulam suggested investigating functions which do not necessarily satisfy a strict rule like ``$f(a+b)=f(a)+f(b)$'', but only satisfy it ``approximately''. Likewise, one can consider a set with an operation which, when applied to two elements of the given set, does not produce an element of this set, but instead produces an element of a larger set which is, in some sense, ``close'' to the original one. This approach can be seen in Definition \ref{DefTao} below (due to T.~Tao, \cite{Tao}), which introduces the notion of an \emph{approximate subgroup} of a group, which then leads to Definition \ref{DefApGr}, for an \emph{approximate group}. At least implicitly, approximate (sub)groups appeared much earlier than in the definition from \cite{Tao} mentioned above. They appeared in connection with compact symmetric identity neighborhoods in locally compact groups, specifically in Lie groups, as well as in theory of mathematical quasi-crystals and in additive combinatorics. \emph{Finite} approximate groups are well investigated -- for instance,
in \cite{BGT}, E.~Breuilard, B.~Green and T.~Tao give the general structure theorem for finite approximate groups.

We will be concerned with infinite countable approximate groups.
The goal of this paper is to establish two results concerning asymptotic dimension of countable approximate groups, based on the work of A.~Dranishnikov and J.~Smith (\cite{DranSmith}). In their paper, Dranishnikov and Smith prove the following theorem for asymptotic dimension of countable groups (\cite[Theorem 2.1]{DranSmith}):
\begin{theorem}\label{Thm: DS sup}
Let $G$ be a countable group. Then $\asdim G = \sup\ \! \asdim F$, where the supremum varies over finitely generated subgroups $F$ of $G$.
\end{theorem}
This theorem is then used, together with Bell-Dranishnikov's Hurewicz type formula for finitely generated groups (\cite[Theorem 7]{BellDran-Hurewicz}), to prove the Hurewicz type formula for all groups (\cite[Theorem 2.3]{DranSmith}):
\begin{theorem}\label{Thm: Hurewicz DS}
Let $f:G\to H$ be a homomorphism of groups with kernel $K$. Then $\asdim G \leq \asdim H + \asdim K$.
\end{theorem}
We are going to prove a version of the Hurewicz type formula for local morphisms of countable approximate groups, namely:
\begin{theorem} 
\label{Thm: local asdim Hurewicz}
 Let $(\Xi, \Xi^\infty)$, $(\Lambda, \Lambda^\infty)$ be countable approximate groups and let $f_6: (\Xi, \Xi^6) \to (\Lambda, \Lambda^6)$ be a  $6$-local morphism. Then
 \[
 \asdim\Xi \leq \asdim\Lambda + \asdim([\ker_6(f_6)]_c).
 \]
\end{theorem}
The proof of this theorem is relying on a theorem by N.~Brodskiy, J.~Dydak, M.~Levin and A.~Mitra (\cite[Theorem 1.2]{BDLM}), quoted here as Theorem \ref{AbstractHurewicz}, which is an asymptotic version of the classical Hurewicz theorem on maps that lower dimension.

As a consequence, we get the Hurewicz type formula for global morphisms of countable approximate groups:
\begin{theorem}\label{CorHurewiczConvenient} Let $(\Xi, \Xi^\infty)$, $(\Lambda, \Lambda^\infty)$ be countable approximate groups and let $f: (\Xi, \Xi^\infty) \to (\Lambda, \Lambda^\infty)$ be a global morphism. Then
 \[
 \asdim\Xi \leq \asdim\Lambda + \asdim([\ker(f)]_c).
 \]
\end{theorem}

Also, based on the method of proving Theorem \ref{Thm: DS sup} (\cite[Theorem 2.1]{DranSmith}), we prove a somewhat analogous result for countable approximate groups, which is a bit more technical:
\begin{theorem}\label{Thm: asdim FG} Let $(\Lambda, \Lambda^\infty)$ be a countable approximate group. Then 
\begin{equation*}\label{eqn: asdim FG}
\asdim\Lambda = {\sup} \left(\underset{\Gamma \text{ finitely generated}}{\bigcup_{\Gamma \leq \Lambda^\infty}}\hspace{-6mm}\{\asdim\Xi\mid \Xi \subset \Lambda^2,\ \Xi \text{ is an approx.\ subgroup of } \Gamma\}\right).
\end{equation*}
\end{theorem}
That is, the asymptotic dimension of $\Lambda$ is the supremum of the asymptotic dimensions of approximate subgroups of finitely generated subgroups of $\Lambda^\infty$, with these approximate subgroups contained in $\Lambda^2$.

Before we turn to the proofs of our results, we provide, in section 2, some basic definitions and facts on approximate groups and local and global morphisms between them.
This is followed by a short reminder section (section 3) focusing on asymptotic dimension and maps in coarse geometry.
Section 4 is introducing asymptotic dimension on countable approximate groups. In section 5, we list several versions of the Hurewicz dimension-lowering mapping theorem, among which are the asymptotic versions -- Theorem \ref{AbstractHurewicz0} (\cite[Theorem 1]{BellDran-Hurewicz}) and Theorem \ref{AbstractHurewicz} (\cite[Theorem 1.2]{BDLM}). Then we prove Theorem \ref{Thm: local asdim Hurewicz} and Theorem \ref{CorHurewiczConvenient} follows. Finally, section 6 is dedicated to the proof of Theorem \ref{Thm: asdim FG}.

Regarding approximate groups, we will be relying heavily on \cite{CHT}, which is a paper by the two authors of this paper together with M.~Cordes.

\section{Approximate groups and morphisms between them}

Before we introduce approximate groups and certain kind of morphisms between them, let us start with some notation and notions that will be needed.

\begin{notation}  Let $A$ and $B$ be subsets of a group $(\Gamma,\cdot)$. Then $AB:=\{a b\ | \ a\in A, b\in B\}$,  so $A^2=AA$, and $A^k=A^{k-1}A$, for $k\in\N_{\geq 2}$. Also, $A^{-1}:=\{a^{-1}\ | \ a\in A\}$, and if $A=A^{-1}$, we say that $A$ is \emph{symmetric}.
We mark the identity element of $\Gamma$ by $e$ or $e_\Gamma$, and if $e\in A$, we say that $A$ is \emph{unital}.
\end{notation}

From now on, unless explicitly stated otherwise, every time we mention a group $\Gamma$, we will mean $(\Gamma, \cdot)$.

\begin{definition}[Filtered group] Let $\Gamma$ be a group and let $\{e\} \subset \Gamma_1 \subset \Gamma_2 \subset \Gamma_3 \subset \dots$ be an ascending sequence of subsets of $\Gamma$. Then $(\Gamma_k)_{k \in \N}$ is called a \emph{filtration} of $\Gamma$ if 
$\Gamma_k\Gamma_l \subset \Gamma_{k+l}, \forall k, l \in \N$.
 The pair  $(\Gamma, (\Gamma_k)_{k \in \N})$ is then called a \emph{filtered group}.\footnote{This type of filtered group is often called an $\N$-filtered group.}
\end{definition}
Note  that if $A$ is a subset of a group $\Gamma$ such that $e\in A$, then $(\Gamma, (A^k)_{k\in \N})$ is a filtered group, referred to as a filtered group \emph{associated with the pair $(A, \Gamma)$}.

\begin{definition}[Maps between subsets of groups]
Let $A$ be a subset of a group $\Gamma$, $B$ be a subset of a group $\Delta$ and let $f: A \to B$ be a map. 
If $A$ is symmetric, i.e., $A=A^{-1}$, we say that $f$ is \emph{symmetric}  if $f(a^{-1}) = f(a)^{-1}$ for all $a \in A$.
If $A$ is unital, i.e., $e_\Gamma\in A$, then we say that $f$ is \emph{unital}  if  $f(e_\Gamma) = e_\Delta$. 

We say that $f: A \to B$ is a \emph{partial homomorphism} if for all $a_1, a_2 \in A$ with $a_1a_2 \in A$ we have $f(a_1a_2) = f(a_1)f(a_2)$. If $A$ is unital (and symmetric), this implies that $f$ is unital (and symmetric).
\end{definition}

\begin{definition}[Filtered morphisms between filtered groups]
If $(\Gamma, (\Gamma_k)_{k \in \N})$ and $(\Delta, (\Delta_k)_{k \in \N})$ are filtered groups, then a group homomorphism $\rho: \Gamma \to \Delta$ is called a \emph{filtered morphism} of filtered groups $(\Gamma, (\Gamma_k)_{k \in \N})$ and $(\Delta, (\Delta_k)_{k \in \N})$ if  $\rho(\Gamma_k) \subset \Delta_k$, for all $k\in\N$. Then, for each $k\in\N$, the restriction $\rho_k:=\rho|_{\Gamma_k}: \Gamma_k \to \Delta_k$ is a partial homomorphism, referred to as the  \emph{$k$-component of $\rho$}. Similarly, if $N \in \N$, then a partial homomorphism $\rho: \Gamma_N \to \Delta_N$ is called an  \emph{$N$-local filtered morphism} of the filtered groups $(\Gamma, (\Gamma_k)_{k \in \N})$ and $(\Delta, (\Delta_k)_{k \in \N})$ if $\rho(\Gamma_k) \subset \Delta_k$, for all $k\leq N$. As above, for $k \leq N$ the restriction of $\rho$ to $\Gamma_k$ is a partial homomorphism from $\Gamma_k$ to $\Delta_k$, and is called the $k$-component of $\rho$.
\end{definition}

Next we would like to introduce the notion of an approximate subgroup of a group. As mentioned before, the idea behind introducing this notion is -- allowing for the result of the group operation between two elements of a subset to be outside of this subset, but still a finite-set-``close'' to it. The following definition is due to T.~Tao (\cite{Tao}):

\begin{definition}[Approximate subgroup of a group]\label{DefTao} Let $\Gamma$ be a group and let $K \in \N$. A subset $\Lambda \subset \Gamma$ is called 
a \emph{$K$-approximate subgroup} of $\Gamma$ if
\begin{enumerate}[({AG}1)]
\item $\Lambda = \Lambda^{-1}$ and $e \in \Lambda$, and
\item there exists a finite subset $F \subset  \Gamma$ such that $\Lambda^2\subset \Lambda  F$ and $|F| = K$.
\end{enumerate}
We say that $\Lambda$ is an \emph{approximate subgroup} if it is a $K$-approximate subgroup for some $K \in \N$. 
\end{definition}

As long as the set $F$ is finite, the exact number of its elements is not going to be important to us.
We leave to the reader to show the following:

\begin{lemma}\label{Lemma: sym F}
If $\Lambda$ is an approximate subgroup of a group $\Gamma$, then there is a finite set $F_\Lambda \subset\Gamma$ such that 
\[
\Lambda^2\subset (F_\Lambda \Lambda) \cap  (\Lambda F_\Lambda) \ \text{ and } \ F_\Lambda \subset \Lambda^3 .
\]
\end{lemma}

\begin{corollary}
If $\Lambda$ is an approximate subgroup of a group $\Gamma$, then the condition of existence of a finite set $F\subset \Gamma$ such that $\Lambda^2\subset \Lambda F$ is equivalent to the existence of a finite set $\widetilde F\subset \Gamma$ such that $\Lambda^2\subset\widetilde F \Lambda$.
\end{corollary}
Therefore, whenever we say that $\Lambda$ is an approximate subgroup of a group $\Gamma$, we can write the second condition as either $\Lambda^2\subset \Lambda F$ or $\Lambda^2\subset F\Lambda$, for a finite set $F\subset \Lambda^3$.

From Lemma \ref{Lemma: sym F} it easily follows that:
\begin{corollary}\label{cor: commensur}
If $\Lambda$ is an approximate subgroup of a group $\Gamma$ and $F_\Lambda$ is a finite set from $\Lambda^3$ such that $\Lambda^2\subset (F_\Lambda \Lambda) \cap  (\Lambda F_\Lambda)$, then
\[
\Lambda^k \subset  (F^l_\Lambda \Lambda^{k-l}) \cap  (\Lambda^{k-l} F^l_\Lambda), \ \text{for all } k>l\geq 1.
\]
\end{corollary}
\begin{remark}\label{rem: needed later}
For the sake of calculations needed later, let us note here that having a statement like $\Lambda^k \subset \Lambda F_\Lambda^{k-1}=\Lambda F'$ means that, in fact, $F'\subset \Lambda^{k+1}$, because we can assume that we take the smallest possible $F'$, and then having, say, $\lambda_1\lambda_2\ldots\lambda_k=\lambda_0f$ implies that this $f=\lambda_0^{-1}\lambda_1\lambda_2\ldots\lambda_k\in \Lambda^{k+1}$.
\end{remark}


Since for an approximate subgroup $\Lambda$ of a group $\Gamma$ there is the smallest subgroup $\Lambda^\infty := \bigcup_{k \in \N} \Lambda^k$ of $\Gamma$ which contains $\Lambda$, this leads to the following definition:

\begin{definition}[Approximate group]\label{DefApGr}
If $\Lambda$ is an approximate subgroup of a group $\Gamma$, then the group $\Lambda^\infty = \bigcup_{k \in \N} \Lambda^k$, which is the smallest subgroup of $\Gamma$ containing $\Lambda$, is called the \emph{enveloping group} of $\Lambda$. The pair $(\Lambda, \Lambda^\infty)$ is called an \emph{approximate group} and the associated filtered group $(\Lambda^\infty, (\Lambda^k)_{k \in \N})$ is called a \emph{filtered approximate group}.

We say that an approximate group $(\Lambda, \Lambda^\infty)$ is \emph{finite} if $\Lambda$ is  finite (but clearly $\Lambda^\infty$ need not be finite).
We say $(\Lambda, \Lambda^\infty)$ is \emph{countable} if $\Lambda$ is countable, which also implies that $\Lambda^\infty$ is countable.

If $(\Xi, \Xi^\infty)$ and $(\Lambda, \Lambda^\infty)$ are approximate groups, then $(\Xi, \Xi^\infty)$ is called an \emph{approximate subgroup} of $(\Lambda, \Lambda^\infty)$  if $\Xi \subset \Lambda$ and $\Xi^\infty \leq \Lambda^\infty$. 
\end{definition}

Even though both Definitions \ref{DefTao} and \ref{DefApGr} mention approximate subgroups, it is easy to distinguish the two notions in their respective contexts. Also note that we will consider groups as approximate groups by identifying $\Gamma$ with $(\Gamma, \Gamma)$. Then $(\Xi, \Xi^\infty)$ is an approximate subgroup of $(\Gamma, \Gamma)$ in the sense of Definition \ref{DefApGr} if and only if $\Xi$ is an approximate subgroup of $\Gamma$ in the sense of Definition \ref{DefTao}. 

It could certainly be argued that the pair notation $(\Lambda,\Lambda^\infty)$ is a bit cumbersome, and that it might be better to just use $\Lambda$ by itself. However, as we shall see in Definition \ref{def: canonical coarse class}, for example, when we need to introduce a ``nice'' metric on $\Lambda$ we will first do so on $\Lambda^\infty$, so we might as well mention $\Lambda^\infty$ in the very definition. On the other hand,
when we define the asymptotic dimension of an approximate group (see Definition \ref{def: asdim-L}), we will write (just) $\asdim \Lambda$.

\begin{example}[Example and non-example of an approximate (sub)group]\label{non-ex}
An easy example of a $2$-approximate subgroup of $(\Gamma,\cdot)=(\Z, +)$ is $\Lambda:=\{-n, -n+1, \ldots , -1, 0, 1, \ldots, n-1, n\}$, for  any $n\in\N$, because $\Lambda +\Lambda =\{ -2n, \ldots , 2n\} \not\subseteq \Lambda$, but for $F=\{-n,n\}$ we get
$\Lambda +\Lambda = \Lambda + F$. Also $\Lambda^\infty =\Z$, and therefore $(\Lambda, \Z)$ is an approximate group. Although, if the number of elements of $F$ is not important as long as $F$ is finite, we can just take $F=\Lambda$ here, since $\Lambda$ is finite. Clearly examples with infinite $\Lambda$ will be more interesting, and we list some of those below.

On the other hand, if we define $\Lambda:=\{2^i \ | \ i\in \Z\}\cup\{0\}\cup\{-2^i \ | \ i\in \Z\}$, then $\Lambda +\Lambda$ contains
$2^n+2^{n+1}=3\cdot 2^n$, $\forall n\in \N$, so it contains infinitely many numbers which are not in $\Lambda$, and we leave it to the reader to show that there is no finite set $F\subseteq \Z$ such that $\Lambda +\Lambda \subseteq\Lambda +F$, i.e.,
 $\Lambda$ is not an approximate subgroup of $\Z$.
\end{example}

A list of some more examples, from \cite{CHT}:
\begin{itemize}

\item If $\Gamma$ is a group and $H\leq \Gamma$ , then $H$ is also an approximate subgroup of $\Gamma$, so the pair $(H,H)$ is an approximate subgroup of $(\Gamma, \Gamma)$.

\item If  $\Gamma$ is a group and $F$ is a finite symmetric subset of $\Gamma$ which contains $e$, then $F$ is clearly an approximate subgroup of $\Gamma$, so $(F, F^\infty)$ is an approximate group.

\item If $\Lambda$ is an approximate subgroup of a group $\Gamma$, then $\Lambda^k$ is also an approximate subgroup of $\Gamma$, so $(\Lambda^k,\Lambda^\infty)$ is an approximate group, for all $k\in\N$.


\item Let ${\rm BS}(1,2) =  \langle a, b \mid bab^{-1} = a^2 \rangle$, i.e., the
Baumslag-Solitar group of type $(1,2)$, and define  $\Lambda :=  \langle a \rangle \cup \{b,b^{-1}\}$. Note that $\Lambda$ is symmetric, contains $e$ and $\Lambda^\infty={\rm BS}(1,2)$. A calculation using the defining relation (and $(b^{-1}ab)^2 = a$) shows that 
$ \Lambda^2 \subseteq \Lambda\cdot \{e, b, b^{-1}, b^{-1}a\},$
hence $(\Lambda, \Lambda^\infty)$ is an approximate group.

\item If $\Gamma$ is a locally compact group and $W$ is a relatively compact\footnote{A set is \emph{relatively compact} if it has compact closure.} symmetric neighborhood of identity $e$ in $\Gamma$, then $(W,W^\infty)$ is an approximate group.

\item 
Cut-and-project construction (see Example 2.87 in \cite{CHT} for more details):
Let $G, H$ be locally compact groups, and denote by $\proj_G: G \times H \to G$ and $\proj_H: G \times H \to H$ the canonical projections. Let $\Gamma$ be a subgroup of $G \times H$ such that the restriction $\proj_G|_\Gamma$ is injective. Then for any relatively compact symmetric identity neighborhood $W$ in $H$, the set
$\Lambda(\Gamma, W) := \proj_G(\Gamma \cap (G \times W))$ is an approximate subgroup of $G$, so $(\Lambda(\Gamma, W), (\Lambda(\Gamma, W))^\infty)$ is an approximate group.
Sets of the form $\Lambda(\Gamma, W)$ are referred to as \emph{cut-and-project sets}, because such a set arises from $\Gamma$ by first cutting it with the ``strip'' $G \times W$ in $G \times H$, and then projecting down to $G$. 
\end{itemize}

On the other hand, unlike in group theory, the intersection of two approximate subgroups need not be an approximate subgroup:
\begin{example}\label{IntersectionFail}
Let $M :=\{(2k)^2\ | \ k\in \N_0 \}\cup \{-(2k)^2\ | \ k\in \N\}$, $\Lambda \coloneqq  (2\Z+1) \cup M$, and $\Gamma \coloneqq  2\Z$.
Since $\Gamma$ is a subgroup of $\Z$, $\Gamma$ is also an approximate subgroup of $\Z$. We leave it to the reader to show that $\Lambda$ is an approximate subgroup of $\Z$, but $M$ is not.
Therefore the intersection $\Lambda \cap \Gamma = M$ is not an approximate subgroup of $\Z$. 
\end{example}

However, as shown in Lemma 2.54 in \cite{CHT}, this is fixed by higher intersections:
\begin{lemma}
 If $\Lambda$ and $\Xi$ are approximate subgroups of a group $\Gamma$ and $k,l \geq 2$, then 
$\Lambda^k \cap \Xi^l$ is an approximate subgroup of $\Gamma$.
\end{lemma}
There are more interesting (and complicated) examples of approximate (sub)groups. E.g., in $\mathrm{AAut}(\cT_d)$, which is the group of almost automorphisms of an infinite $d$-regular rooted tree, a family of elements of depth $\leq m$ is an approximate subgroup (\cite[Example 2.30]{CHT}).
But let us move on to maps, that is,
let us now introduce certain types of morphisms between approximate groups.

\begin{definition}[Global and local morphism of approximate groups]\label{def: global morphism} 
A \emph{global morphism} $\rho: (\Xi, \Xi^\infty) \to (\Lambda, \Lambda^\infty)$ between approximate groups is a filtered morphism between the associated filtered groups, i.e., it is a group homomorphism $\rho: \Xi^\infty \to \Lambda^\infty$ which restricts to a partial homomorphism $\rho_1:= \rho|_\Xi: \Xi \to \Lambda$. Let us emphasize here that each of its $k$-components $\rho_k:=\rho|_{\Xi^k} :\Xi^{k}\to \Lambda^k$ is a partial homomorphism.

Similarly, given an  $N \in \N$, an \emph{$N$-local morphism}
 between approximate groups  $(\Xi, \Xi^\infty)$ and $(\Lambda, \Lambda^\infty)$ is an $N$-local filtered morphism between the associated filtered groups, i.e., it is a partial homomorphism $\rho_N: \Xi^N \to \Lambda^N$ which restricts to a partial homomorphism $\rho_1:=\rho_N|_\Xi: \Xi \to \Lambda$. We often write $\rho_N:(\Xi, \Xi^N)\to (\Lambda, \Lambda^N)$ for it. Again, note that  each of its $k$-components $ \rho_k:=\rho_N|_{\Xi^k} :\Xi^{k}\to \Lambda^k$, for $k\leq N$, is a partial homomorphism.
\end{definition}

\begin{definition}[Partial kernels and images]\label{KernelsAndImages} Let $\rho: (\Xi, \Xi^\infty) \to (\Lambda, \Lambda^\infty)$ be a global morphism of approximate groups. We define the \emph{$k$th partial kernel} and the \emph{$k$th partial image} of $\rho$ as $\ker_k(\rho) \coloneqq  \ker(\rho) \cap \Xi^k=\rho^{-1}(e_\Lambda)\cap \Xi^k$, respectively $\im_k(\rho) \coloneqq  \rho(\Xi^k) = \rho(\Xi)^k$. By definition these are the partial kernels and images of the induced map between the associated filtered groups. Clearly $(\ker_k(\rho))_{k\in\N}$ and $(\im_k(\rho))_{k\in\N}$ are ascending sequences of subsets.

We also define partial kernels and partial images for an $N$-local morphism $\rho_N:(\Xi, \Xi^N)\to (\Lambda, \Lambda^N)$ in the analogous way, for $k\leq N$: $\ker_k(\rho_N) \coloneqq  \rho_N^{-1}(e_\Lambda)\cap \Xi^k$, and $\im_k(\rho_N) \coloneqq  \rho_N(\Xi^k) = \rho_N(\Xi)^k$. Again,
we have $\ker_1(\rho_N)\subset \ker_2(\rho_N)\subset\ldots\subset \ker_N(\rho_N)$ and  $\im_1(\rho_N)\subset \im_2(\rho_N)\subset\ldots\subset \im_N(\rho_N)$.
\end{definition}

Note that the first partial kernel of a global morphism of approximate groups need not be an approximate subgroup, as the following example from \cite{CHT} shows:

\begin{example}
\label{NotPKG}
Let $M$ and $\Lambda$ be as in Example \ref{IntersectionFail}. Then the canonical projection $\Z \to \Z/2\Z$ induces a global morphism
$\rho: (\Lambda, \Z) \to (\Z/2\Z, \Z/2\Z) $ of approximate groups, and  $\ker_1(\rho) = \rho^{-1}(0)\cap \Lambda= 2\Z \cap \Lambda= M$, which is not an approximate subgroup of $\Z$.
\end{example}
However, for $k\geq 2$, the $k$th partial kernels of a global morphism are, in fact, approximate subgroups, and so are the $k$th partial kernels of an $N$-local morphism,  for $N\geq 10$ and  $k\in \{2, \cdots \lfloor \frac{N+2}{6} \rfloor\}$ (see Corollary 2.57 and Lemma 2.59 in \cite{CHT}):
\begin{lemma}\label{lem: partial kernel}
Let  $(\Xi, \Xi^\infty)$ and $(\Lambda, \Lambda^\infty)$ be approximate groups.
\begin{enumerate}[(i)]
\item If $\rho: (\Xi, \Xi^\infty) \to (\Lambda, \Lambda^\infty)$ is a global morphism, then the partial kernels $\ker_k(\rho) =   \ker(\rho) \cap \Xi^k$ are approximate subgroups of $\Xi^\infty$, for all $k \geq 2$.

\item If $\rho_N: (\Xi, \Xi^{N}) \to (\Lambda, \Lambda^{N})$ is an $N$-local morphism of approximate groups for some $N\geq 10$, then 
$\ker_2(\rho_{N}),\ \dots , \ \ker_{\lfloor \frac{N+2}{6} \rfloor}(\rho_{N})$ are approximate subgroups of $\Xi^\infty$. 
\end{enumerate}
\end{lemma}


\section{Some background on ``coarse'' maps and asymptotic dimension}\label{AsdimPrelim}

Let us review some well-known tools from coarse geometry: coarse equivalences and related notions for maps, as well as asymptotic dimension. 
The terminology we are using for maps is from \cite{CdlH}.
We will be dealing with metric spaces, so let us recall that a metric space is \emph{proper} if its closed balls are compact. 
In a metric space $(X,d)$ we will use notation $B(x, r)$ for an open ball, and $\overline{B}(x, r)$ for a closed ball with center at the point x and with radius $r>0$. If $A\subseteq X$ and $R>0$,  then $N_R(A)$ will refer to an open $R$-neighborhood of a set $A$ in $(X,d)$.

\subsection{Maps}
The words \emph{function} and \emph{map} will mean the same thing\footnote{There are papers in topology in which a \emph{map} is understood to be a continuous function, but that is not the case in this paper.}, and will be used interchangeably.
\begin{definition}
Let $(X, d_X)$, $(Y, d_Y)$ be metric spaces. A map $f:X\to Y$  is called:
\begin{enumerate}[(i)]
\item \emph{proper} if pre-images of compact sets are compact,
\item \emph{coarsely proper} if pre-images of bounded sets are bounded.
\end{enumerate}
\end{definition}

\begin{definition}\label{def: coarsely - maps} 
Let $(X, d_X)$, $(Y, d_Y)$ be metric spaces and let $f: X \to Y$ be a map. Let $\Phi_-, \ \Phi_+: [0, \infty) \to [0, \infty)$ be non-decreasing functions with $\lim_{t \to \infty} \Phi_i(t) = \infty$, for $i=-,+$.
\begin{enumerate}[(i)]
\item If $d_Y(f(x), f(x')) \leq \Phi_+(d_X(x,x'))$ for all $x,x'\in X$, then $f$ is called \emph{coarsely Lipschitz}, and $\Phi_+$ is referred to as an \emph{upper control function} for $f$.

\item If $d_Y(f(x), f(x')) \geq \Phi_-(d_X(x,x'))$  for all $x,x'\in X$, then $f$ is called \emph{coarsely expansive}, and $\Phi_-$ is referred to as a \emph{lower control function} for $f$.

\item If $\Phi_-(d_X(x,x')) \leq d_Y(f(x), f(x')) \leq \Phi_+(d_X(x,x'))$ for all $x,x'\in X$, then $f$ is called a \emph{coarse embedding}.

\item If $f$ is a coarse embedding which is also \emph{coarsely surjective}, i.e., if there is an $L>0$ such that $N_L(f(X))=Y$, then $f$ is called a \emph{coarse equivalence}.
\end{enumerate}
\end{definition}

 We will mark the existence of a coarse equivalence between two metric spaces $(X,d_X)$ and $(Y,d_Y)$ by $X\overset{CE}{\approx}Y$.
To any metric space $(X,d_X)$ we can assign its coarse equivalence class:
\[
[(X,d_X)]_c:=\{ (Y, d_Y)\ |\ (Y, d_Y) \text{  is a metric space such that } (Y,d_Y)\overset{CE}{\approx} (X, d_X)\}.
\]

In the special case of (iii) and (iv) from Definition \ref{def: coarsely - maps} when both control functions $\Phi_-$ and $\Phi_+$ can be taken to be affine, the map $f$ is called a \emph{quasi-isometric embedding} (or QI-embedding), and a \emph{quasi-isometry}, respectively.

Some basic properties of coarsely Lipschitz and coarsely expansive maps can be found in \cite[Chapter 3]{CdlH}. Note that the terminology varies a lot throughout the literature (see \cite[Remark 3.A.4]{CdlH}): in \cite{Roe}, coarsely Lipschitz maps are referred to as (uniformly) bornologous, while in \cite{BDLM} they are called large scale uniform; in \cite{BellDran1}, coarsely expansive maps are called uniformly expansive maps; in \cite{DranSmith}, coarse surjectivity of $f$ (with constant $L$) is called $L$-density of the image of $f$ in $Y$. Coarsely Lipschitz maps and coarsely expansive maps can also be characterized without explicit mentioning of control functions (see \cite[Prop.\ 3.A.5]{CdlH}):
\begin{lemma} \label{characterisation-CL,CE}
A map $f:X\to Y$ between metric spaces is coarsely Lipschitz if and only if for all $r\geq 0$, there exists $s\geq 0$ such that, if $x, x' \in X$ satisfy $d_X(x, x')\leq r$, then $d_Y (f(x),f(x'))\leq s$. 

A map $f:X\to Y$ between metric spaces is coarsely expansive if and only if for all $s\geq 0$, there exists $r\geq 0$ such that, if $x, x' \in X$ satisfy $d_X(x, x')\geq r$, then $d_Y (f(x),f(x'))\geq s$. 
\end{lemma}


\subsection{Asymptotic dimension}
\begin{definition}\label{def: asdim} Let $n \in \mathbb N_0$.
A metric space $(X,d)$ has \emph{asymptotic dimension}\index{asymptotic dimension}\index{$\asdim$} $\asdim X \leq n$ if for every $R >0$ there exist $n+1$ collections $\mathcal U^{(0)}, \dots, \mathcal U^{(n)}$ of subsets of $X$ with the following properties:
\begin{enumerate}[({ASD}1)]
\item $\cU=\bigcup_{i=0}^n \mathcal U^{(i)}$ is a cover of $X$;
\item $\cU$ is \emph{uniformly bounded}\index{uniformly bounded family}, i.e., there exists $D>0$ such that $\diam U \leq D$, for all $U \in \mathcal U^{(i)}$ and all $i=0, \dots, n$;
\item For every $i=0, \dots, n$ the collection $\mathcal U^{(i)}$ is \emph{$R$-disjoint}\index{$R$-disjoint}, i.e., whenever $U,V \in \mathcal U^{(i)}$ are such that $U\neq V$, then $d(U, V) \geq R$.
\end{enumerate}
We say that $\asdim X = n$ if 
this $n\in \N_0$ is the smallest number so that these three properties are satisfied for every $R>0$.
If there is no such $n\in \N_0$, we say that $\asdim X = \infty$.
\end{definition}
It evidently suffices for $(X,d)$ to have the above properties for all sufficiently large $R$. This definition of asymptotic dimension is one of several equivalent definitions, see \cite{BellDran1}. It is sometimes referred to as the \emph{coloring definition} of $\asdim$, because each of the collections $\mathcal U^{(0)}, \dots, \mathcal U^{(n)}$ can be regarded as having its own color $i$. The following lemma lists some important properties of $\asdim$ (see, for example, \cite{BellDran1}):

\begin{lemma}\label{AsdimInvariant} Let $(X,d_X)$ and $(Y,d_Y)$ be metric spaces.
\begin{enumerate}[(i)]
\item If $A$ is a subspace of $X$, then $\asdim A\leq \asdim X$.
\item If $X$ and $Y$ are coarsely equivalent, then  $\asdim X = \asdim Y$.
\item If $X$ embeds coarsely into $Y$, then $\asdim X \leq \asdim Y$. 
\end{enumerate}
\end{lemma}

Let us emphasize that Lemma \ref{AsdimInvariant}(ii) states that $\asdim$ is a coarse invariant, because it is preserved by coarse equivalences.

\section{Asymptotic dimension of countable approximate groups}

We will be focusing on countable approximate groups, i.e., on such approximate groups $(\Lambda,\Lambda^\infty)$ for which $\Lambda$ is countable and therefore its enveloping group $\Lambda^\infty$ is countable, too. We will always consider countable groups as discrete groups.
Let us recall some facts about choosing a ``nice'' metric on a countable group which need not be finitely generated, so that this metric agrees with discrete topology on the group.

\begin{remark}[See \cite{DranSmith}, Section 1]\label{rem: weight}
Let $\Gamma$ be a countable discrete group and let $S$ be a symmetric (possibly infinite) generating set of $\Gamma$.
One can define a function $w: S\cup\{e\}\to [0, \infty)$ so that it is proper and satisfies $w^{-1}(0) = \{e\}$ and $w(s) = w(s^{-1})$ for all $s \in S$. This function $w$ is then called a \emph{weight function on $S$}. It turns out (see Proposition 1.3 of \cite{DranSmith}) that every weight function $w: S \cup\{e\}\to [0, \infty)$ defines a proper norm on $\Gamma$ by
\[
\|g\|_{S, w} : = \inf\left\{\sum_{i=1}^n w(s_i)\mid g = s_1 \cdots s_n, \; s_i \in S\right\}.
\]
The associated metric $d_{S,w}$, given by $d_{S,w}(x,y):=\|x^{-1}y\|_{S,w}$, is left-invariant and proper, and it induces the (already given) discrete topology on $\Gamma$ (see \cite[Section 1]{Smith} for a short explanation of this).
\end{remark}

Since this metric is proper, closed balls in $(\Gamma, d_{S,w})$ are compact and so they are finite sets, since $\Gamma$ is discrete. Therefore open balls with bounded radii are also finite sets.

If $S$ is finite, i.e., if $\Gamma$ is finitely generated, then we may choose $w(s)=1, \forall s\in S$. In this case, $d_S \coloneqq  d_{S,1}$ is exactly the \emph{word metric} on $\Gamma$ with respect to the (finite) generating set $S$.

Another important fact (\cite[Prop. 1.1]{DranSmith}), stated for countable groups:
\begin{proposition} If $d_1$ and $d_2$ are two left-invariant proper metrics on the same countable group $\Gamma$, then the identity map between $(\Gamma, d_1)$ and $(\Gamma, d_2)$ is a coarse equivalence. 
\end{proposition}

In light of the previous story, we can introduce:
\begin{definition}
The \emph{canonical coarse class $[\Gamma]_c$} of a countable group $\Gamma$ is the coarse equivalence class
of the metric space $(\Gamma, d)$, where $d$ is some (hence any) left-invariant proper metric on $\Gamma$. That is, 
\[ [\Gamma]_c:=[(\Gamma, d)]_c=\{(X,d')\ |\ (X,d') \text{ is a metric space such that } (X,d')\overset{CE}{\approx} (\Gamma, d)\}.
\]
\end{definition}
In particular, $(\Gamma, d_{S,w})$  from Remark \ref{rem: weight} represents $[\Gamma]_c$.

\bigskip

Now we define the asymptotic dimension of a countable group as follows (\cite[Section 2]{DranSmith}):
\begin{definition}\label{def: asdim - group}
For a countable group $\Gamma$, its asymptotic dimension is defined as  
\[\asdim \Gamma := \asdim \ \! (\Gamma, d),\] 
where $d$ is any left-invariant proper metric on $\Gamma$. We can also define $\asdim \ \! ([\Gamma]_c) := \asdim \ \! (\Gamma, d)$, so $\asdim \Gamma =\asdim \ \! (\Gamma, d)=\asdim \ \! ([\Gamma]_c)$.
\end{definition}

We now want to define the canonical coarse class and asymptotic dimension for a countable approximate group. We will do so following the process described in the beginning of Chapter 3 of \cite{CHT}. To begin with, the first part of Lemma 3.1 of \cite{CHT} states:
\begin{lemma}\label{ExternalQIType} Let $\Gamma$ be a countable group, $A \subseteq \Gamma$ be a subset and $d$ and $d'$ be left-invariant proper metrics on $\Gamma$. Then the identity map $(A, d|_{A \times A}) \to (A, d'|_{A \times A})$ is a coarse equivalence. 
\end{lemma}

Because of this lemma, the following notions are well-defined, i.e., they are independent of the choice of a left-invariant proper metric $d$:
\begin{definition} \label{def: canonical coarse class}
Let $(\Lambda, \Lambda^\infty)$ be a countable approximate group and let $d$ be a left-invariant proper metric on the countable group $\Lambda^\infty$. Then for any subset $A \subseteq \Lambda^\infty$, we define the \emph{cannonical coarse class of $A$} by
\[
[A]_c \coloneqq  [(A, d|_{A \times A})]_c = \{(X,d')\ |\ (X,d') \text{ is a metric space such that } (X,d')\overset{CE}{\approx} (A, d|_{A \times A})\}.
\]
In particular, this defines the canonical coarse class of $\Lambda$,
$[\Lambda]_c \coloneqq  [(\Lambda, d|_{\Lambda \times \Lambda})]_c$.
\end{definition}

\medskip

Also, if $\Lambda$ is an approximate subgroup of a countable group $\Gamma$, and if $d$ is a left-invariant proper metric on $\Gamma$, then $d|_{\Lambda^\infty \times \Lambda^\infty}$ is a left-invariant proper metric on $\Lambda^\infty$, and hence
$[\Lambda]_c =  [(\Lambda, d|_{\Lambda \times \Lambda})]_c.$
Thus the canonical coarse class of $\Lambda$ is independent of the ambient group used to define it. 
Moreover, since the restriction of a left-invariant proper metric on $\Lambda^\infty$ is still a proper metric on $\Lambda$, the class $[\Lambda]_c$ admits a representative which is a proper metric space. 

\medskip

The following two facts
record some basic properties of the canonical coarse classes of approximate groups.  

\begin{lemma}\label{LambdaKCoarse}  Let $\Gamma$ be a countable group and $\Lambda \subset \Gamma$ be an approximate subgroup. Then $\Lambda^k\overset{CE}{\approx} \Lambda$ for every $k \geq 1$, so
 $[\Lambda^k]_c = [\Lambda]_c$ for every $k \geq 1$.
\end{lemma}
\begin{proof}
This easily follows from Corollary \ref{cor: commensur}: there is a finite set $F$ from $\Lambda^3\subset \Lambda^\infty$ such that, for every $k\geq 2$,
$\Lambda^k\subset \Lambda^{k-1}F\subset \ldots \subset \Lambda F^{k-1}$, which means that $\Lambda^k$ is within a finite distance of $\Lambda$, making them coarsely equivalent.
\end{proof}

\begin{lemma}
\label{SubgroupCoarse} Let $(\Xi, \Xi^\infty)$ be an approximate subgroup of a countable approximate group $(\Lambda, \Lambda^\infty)$. Then the inclusion $\Xi\hookrightarrow \Lambda$ induces an embedding  $[\Xi]_c \subseteq [\Lambda]_c$. 
\end{lemma}

Now let us record some facts involving coarse classes of approximate groups and (local) morphisms.


\begin{lemma}\label{Bornologous}
Let $(\Xi, \Xi^\infty)$, $(\Lambda, \Lambda^\infty)$ be countable approximate groups and let $f_2: (\Xi, \Xi^2) \to (\Lambda, \Lambda^2)$ be a $2$-local morphism. Then the restriction $f_1: \Xi \to \Lambda$ is coarsely Lipschitz, with respect to the canonical coarse classes of $\Xi$ and $\Lambda$.
\end{lemma}
\begin{proof}
First, recall that since $\Xi$ and $\Lambda$ are symmetric and unital, then the partial homomorphisms $f_1:\Xi \to \Lambda$ and $f_2:\Xi^2 \to \Lambda^2$ are both symmetric and unital.

Regarding the canonical coarse classes of $\Xi$ and $\Lambda$, fix some left-invariant proper metrics $d$ on $\Xi^\infty$ and $d'$ on $\Lambda^\infty$. In addition, fix a random $t>0$.
Since $d$ is proper, the closed ball   $\overline B_d(e_\Xi, t)$ is finite, and so is $\overline B_d(e_\Xi, t)\cap \Xi^2$, so there exists an $s>0$ such that $f_2(\overline B_d(e_\Xi, t)\cap \Xi^2)\subseteq \overline B_{d'}(e_\Lambda, s)\cap \Lambda^2$.

If we take any $g, h \in \Xi$ with $d(g,h)\leq t$, then left-invariance of $d$ implies that $d(h^{-1}g,e_\Xi)\leq t$, so
\[ d'\left(f_2(h^{-1}g),f_2(e_\Xi)\right)=d'\left(f_2(h)^{-1}f_2(g),e_\Lambda\right)\leq s,
\]
so by left-invariance of $d'$ we have $d'\left(f_2(g),f_2(h)\right)=d'\left(f_1(g),f_1(h)\right) \leq s$.
Therefore, by Lemma \ref{characterisation-CL,CE}, $f_1$ is coarsely Lipschitz.
\end{proof}

Next we would like to see how the partial kernels of global (and local) morphisms of countable approximate groups relate to each other, and what this means for their canonical coarse classes. 
Let us quote here 
a part of the statement of Lemma 2.59 of \cite{CHT}, 
adjusted for our purposes, and also quote its proof for the sake of convenience.

\begin{lemma}\label{lem: coarse equiv ker}
Let $(\Xi,\Xi^\infty)$ and $(\Lambda, \Lambda^\infty)$ be approximate groups such that $(\Xi,\Xi^\infty)$ is countable. Let $\rho_{k+1}: (\Xi,\Xi^{k+1})\to (\Lambda,\Lambda^{k+1})$ be a $(k+1)$-local morphism, for $k\geq 2$. Then there is an $R>0$ such that $\ker_2(\rho_{k+1})\subset \ker_j(\rho_{k+1})\subset N_R(\ker_2(\rho_{k+1}))$, for all $j=3, \ldots , k$, which implies that $\ker_j(\rho_{k+1}) \overset{CE}{\approx} \ker_2(\rho_{k+1})$,  for all $j=3, \ldots , k$.
\end{lemma}
\begin{proof}
Let us assume, without loss of generality, that $\rho_1:\Xi\to\Lambda$ is surjective, and let us fix a left-invariant proper metric $d$ on $\Xi^\infty$.
It is enough to show that there exists an $R >0$ such that 
$\ker_k(\rho_{k+1})\subset N_R(\ker_2(\rho_{k+1}))$. 
Since $(\Xi, \Xi^\infty)$ is an approximate group, by Remark \ref{rem: needed later} we can
choose a finite set $F \subset \Xi^{k+1}$ such that $\Xi^k \subset \Xi F$ and set $C_1 \coloneqq  \max_{f \in F}d(f,e_\Xi)$. 
Given a $\xi \in \ker_k(\rho_{k+1})= \rho_{k+1}^{-1}(e_\Lambda)\cap \Xi^k$, we choose $\xi_o \in \Xi$ and $f \in F$ such that $\xi = \xi_o f$. Then $d(\xi, \xi_o) = d(\xi_of, \xi_o) = d(f, e_\Xi) \leq C_1$. 
Moreover, since $\rho_{k+1}$ is a $(k+1)$-local morphism and $\xi_o$, $f$, $\xi \in \Xi^{k+1}$, we have $e_\Lambda = \rho_{k+1}(\xi) = \rho_{k+1}(\xi_o)\rho_{k+1}(f)$, and hence
\[
\rho_1(\xi_o)  = \rho_{k+1}(\xi_o)= (\rho_{k+1}(f))^{-1} \in (\rho_{k+1}(F))^{-1} \cap \Lambda.
\]
Thus if we define $F' \coloneqq  (\rho_{k+1}(F))^{-1}$, i.e., the set of inverses of elements of $\rho_{k+1}(F)$,
 then $\xi_o \in \rho_1^{-1}(F' \cap \Lambda)$ and hence
\begin{equation}\label{CoarseKernel1}
\ker_k(\rho_{k+1}) \subset N_{C_1}(\rho_1^{-1}(F' \cap \Lambda)),
\end{equation}
where $N_{C_1}$ denotes the $C_1$-neighborhood with respect to $d$. Since $F$ is finite, so is $F'$ and so is $F' \cap \Lambda$, say $F' \cap \Lambda  = \{\lambda_1, \dots, \lambda_N\}$. Since $\rho_1$ is surjective, we find $\xi_1, \dots, \xi_N \in \Xi$ with $\rho_1(\xi_i)=\lambda_i$ for all $i=1, \dots, N$. Now let $\xi_i' \in \rho_1^{-1}(\lambda_i)$ and set $\xi_i'' \coloneqq  \xi_i'\xi_i^{-1}$. Since $\xi_i', \xi_i^{-1}\in \Xi$, we have $\xi_i'' \in \Xi^2$, and so $\xi_i, \xi_i', \xi_i'' \in \Xi^{k+1}$. We deduce that $\rho_{k+1}(\xi_i'')=\rho_{k+1}(\xi_i')\rho_{k+1}(\xi_i^{-1})=\lambda_i\lambda_i^{-1}=e_\Lambda$, and hence 
$\xi_i'' \in \rho_{k+1}^{-1}(e_\Lambda) \cap\Xi^2= \ker_2(\rho_{k+1})$. On the other hand, if we set $C_2 \coloneqq  \max_{i=1, \dots, N} d(\xi_i, e_\Xi)$, then $d(\xi_i', \xi_i'') = d(\xi_i', \xi_i'\xi_i^{-1}) = d(e_\Xi, \xi_i^{-1}) \leq C_2$. This shows that
\begin{equation}\label{CoarseKernel2}
\rho_1^{-1}(F' \cap \Lambda)= \bigcup_{i=1}^N \rho_1^{-1}(\lambda_i) \subset N_{C_2} (\ker_2(\rho_{k+1}) ),
\end{equation}
and combining \eqref{CoarseKernel1} and \eqref{CoarseKernel2} we obtain
$\ker_k(\rho_{k+1}) \subset N_{C_1+C_2}(\ker_2(\rho_{k+1}) )$, which finishes the proof.
\end{proof}

\begin{corollary}\label{cor: coarse equiv ker}
If $\rho: (\Xi,\Xi^\infty) \to(\Lambda, \Lambda^\infty)$ is a global morphism of approximate groups where $(\Xi,\Xi^\infty)$ is countable, then its partial kernels $\ker_2(\rho), \ker_3(\rho), \ldots , \ker_k(\rho), \ker_{k+1}(\rho), \ldots$ are all mutually coarsely equivalent.
\end{corollary}

\begin{remark}[Coarse kernel of a global morphism]\label{CoarseKernel} 
Let $(\Xi, \Xi^\infty)$, $(\Lambda, \Lambda^\infty)$ be approximate groups with $(\Xi,\Xi^\infty)$ countable, and let $\rho_k: (\Xi, \Xi^k) \to (\Lambda, \Lambda^k)$ be a $k$-local morphism, for $k\geq 3$.
Then by Lemma \ref{lem: coarse equiv ker} we have
\[
[\ker_2(\rho_k)]_c = [\ker_3(\rho_k)]_c = \dots = [\ker_{k-1}(\rho_{k})]_c.
\]
According to Lemma \ref{lem: partial kernel}, we can only guarantee that for $k\geq 10$, $\ker_2(\rho_k), \ldots, \ker_{\lfloor \frac{k+2}{6} \rfloor}(\rho_{k})$ are approximate subgroups of $\Xi^\infty$, which is not an impediment to defining the coarse classes of higher partial kernels.

By Corollary \ref{cor: coarse equiv ker},
for a global morphism $\rho:(\Xi, \Xi^\infty)\to (\Lambda, \Lambda^\infty)$ we have 
\[
[\ker_2(\rho)]_c = [\ker_3(\rho)]_c = \dots = [\ker_k(\rho)]_c = [\ker_{k+1}(\rho)]_c=\ldots,
\]
where this common coarse equivalence class will be called the \emph{coarse kernel} of $\rho$ and denoted $[\ker(\rho)]_c$. 
Note that for a global morphism $\rho$, all partial kernels with index $\geq 2$ are approximate subgroups of $\Xi^\infty$, by Lemma \ref{lem: partial kernel}.
\end{remark}

\medskip

Finally, we define the asymptotic dimension of a countable approximate group as follows:
\begin{definition}\label{def: asdim-L}
For a countable approximate group $(\Lambda, \Lambda^\infty)$, its asymptotic dimension is defined as
\[\asdim \Lambda := \asdim \ \! (\Lambda, d|_{\Lambda \times \Lambda}),
\]
where $d$ is any left-invariant proper metric on $\Lambda^\infty$.

Also, if $A$ is any subset of $\Lambda^\infty$,  then the asymptotic dimension of $A$ is defined as $\asdim A := \asdim \ \! (A, d|_{A \times A})$.
\end{definition}
Note that for any $A\subseteq \Lambda^\infty$ (including $A=\Lambda$), we can also say that $\asdim A = \asdim X$, where $X$ is any metric space from $[A]_c$, as defined in Definition \ref{def: canonical coarse class}.

In addition, like in Definition \ref{def: asdim - group}, we can define $\asdim \ \! ([A]_c)$ to be equal to $\asdim A$, and, in particular,
$\asdim \ \! ([\Lambda]_c)$ to be equal to $\asdim \Lambda$.

\medskip

From Definition \ref{def: asdim-L} and Lemmas \ref{LambdaKCoarse}, \ref{SubgroupCoarse} and \ref{AsdimInvariant} it follows that:
\begin{corollary}
\label{CorMonotonicity} Let $(\Lambda,\Lambda^\infty)$ be a countable approximate group, and let $(\Xi, \Xi^\infty)  \subset (\Lambda,\Lambda^\infty)$ be an approximate subgroup. Then the following hold:
\begin{enumerate}[(i)]
\item $\asdim\Lambda \leq \asdim \Lambda^\infty$,
\item $\asdim\Lambda^k = \asdim\Lambda$ for all $k \geq 1$,
\item $\asdim\Xi  \leq \asdim\Lambda$.
\end{enumerate}
\end{corollary}


\section[Asymptotic Hurewicz mapping theorem]{An asymptotic Hurewicz dimension-lowering mapping theorem for countable approximate groups}

In this section we prove Theorem \ref{Thm: local asdim Hurewicz}, from which  Theorem \ref{CorHurewiczConvenient} will follow directly. To motivate our result, we first recall the classical Hurewicz theorem for dimension-lowering maps (e.g.~\cite[Theorem 4.3.4]{Engelking}). 
\begin{theorem}[Hurewicz dimension-lowering theorem] Let $X$ and $Y$ be me\-tri\-zable spaces and let $f: X \to Y$ be a closed map. Then
\[
\dim X \leq \dim Y + \dim (f), \quad \text{where} \quad \dim(f) \coloneqq  \sup\{ \dim(f^{-1}(y))\ | \  y\in Y\}.
\]
\end{theorem}
In order to state a version of this theorem for asymptotic dimension, we need a definition from \cite{BellDranOnAsdimOfGroups}:

\begin{definition} Let $\mathbb Y \coloneqq  \{Y_\alpha\}_{\alpha \in A}$ be a collection of metric spaces and let $n\in \N_0$. We say that the asymptotic dimension of $\mathbb Y$ is \emph{uniformly bounded by $n$}, \index{asymptotic dimension!uniformly bounded} and write ${\asdim}(\mathbb Y) \overset{u}{\leq} n$, if for any $R > 0$ there exists $D>0$ such that for every $\alpha \in A$ there exist collections $\mathcal U^{(0)}_\alpha, \dots, \mathcal U^{(n)}_\alpha$ of subsets of $Y_\alpha$ satisfying properties (ASD1)--(ASD3) from Definition \ref{def: asdim} with respect to $R$ and $D$.
\end{definition}
The following asymptotic dimension version of Hurewicz dimension-lowering mapping theorem is due to Bell and Dranishnikov 
(see \cite[Theorem 1]{BellDran-Hurewicz}, or \cite[Theorem 29]{BellDran1}):
\begin{theorem}[Asymptotic Hurewicz mapping theorem, first version]\label{AbstractHurewicz0} Let $f: X \to Y$ be a  Lipschitz map from a geodesic metric space $X$ to a metric space $Y$.  If for every $\rho>0$ the collection
$\mathbb Y_\rho \coloneqq  \{f^{-1}(B(y, \rho))\}_{y \in Y}$ satisfies 
${\asdim}(\mathbb Y_\rho)  \overset{u}{\leq} n$,
then $\asdim X \leq \asdim Y + n.$
\end{theorem}

A generalization of this result, due to Brodskiy, Dydak, Levin and Mitra (\cite[Theorem 1.2]{BDLM}), states:
\begin{theorem}[Asymptotic Hurewicz mapping theorem, second version]\label{AbstractHurewicz} Let $h: X \to Y$ be a coarsely Lipschitz map between metric spaces.  If for every $\rho>0$ the collection
$\mathbb Y_\rho \coloneqq  \{h^{-1}(B(y, \rho))\}_{y \in Y}$ satisfies
${\asdim}(\mathbb Y_\rho)  \overset{u}{\leq} n$,
then $\asdim X \leq \asdim Y + n.$
\end{theorem}
Note that in the original statement of this theorem, what was called a ``large scale uniform'' map is precisely what is called a coarsely Lipschitz map in this paper (in Definition \ref{def: coarsely - maps}). Also, instead of $\asdim X \leq \asdim Y + n$, in \cite[Theorem 1.2]{BDLM} it is written  $ \asdim X \leq \asdim Y +\asdim h$, where $\asdim h$ is defined  as $\sup \{\asdim A \mid A\subseteq X \text{ and } \asdim (h(A))=0\}$. However, the property $ \asdim(\mathbb Y_\rho)  \overset{u}{\leq} n$ can be shown to be equivalent to the map $h$ having an $n$-dimensional control function (terminology of \cite{BDLM}), and by Corollary 4.10 of \cite{BDLM}, this is equivalent to $\asdim h \leq n$.
As stated in the Introduction, as a consequence of Theorem \ref{AbstractHurewicz} we obtain the proof of Theorem \ref{Thm: local asdim Hurewicz}, whose statement we will repeat here for convenience.

\begin{usethmcounterof}{Thm: local asdim Hurewicz}
[Asymptotic Hurewicz mapping theorem for a local morphism of countable approximate groups] 
 Let $(\Xi, \Xi^\infty)$, $(\Lambda, \Lambda^\infty)$ be countable approximate groups and let $f_6: (\Xi, \Xi^6) \to (\Lambda, \Lambda^6)$ be a  $6$-local morphism. Then
 \[
 \asdim\Xi \leq \asdim\Lambda + \asdim([\ker_6(f_6)]_c).
 \]
\end{usethmcounterof}

\begin{proof} Replacing $(\Lambda, \Lambda^6)$ by the image of $f_6$ if necessary, we may assume that $f_6$ is surjective as a map of pairs. We fix left-invariant proper metrics $d$ on $\Xi^\infty$ and $d'$ on $\Lambda^\infty$. Given $\xi \in \Xi^\infty$ and $\lambda \in \Lambda^\infty$, we denote by $B(\xi, \rho)$ and $B'(\lambda, \rho)$ the open balls of radius $\rho$ around $\xi$, respectively $\lambda$, in $(\Xi^\infty,d)$, respectively $(\Lambda^\infty, d')$. Also, let $e:=e_\Lambda$ denote the identity in $\Lambda\subseteq \Lambda^\infty$, while $e_\Xi$ denotes the identity in $\Xi\subseteq \Xi^\infty$. Then $\ker_6(f_6)=f_6^{-1}(e)$. Also,
by left-invariance of $d'$ we have that $B'(\lambda, \rho)=\lambda B'(e, \rho)$, for any $\lambda\in\Lambda^\infty$.

By Lemma \ref{Bornologous} the restriction $f_1: \Xi \to \Lambda$ is coarsely Lipschitz. In view of Theorem \ref{AbstractHurewicz} it thus suffices to show that, for every $\rho>0$, the collection
\[
\mathbb Y_\rho \coloneqq  \{f_1^{-1}(B'(\lambda, \rho) \cap \Lambda)\}_{\lambda \in \Lambda}
\]
satisfies the inequality
\begin{equation}\label{HurewiczToShow1}
{\asdim}(\mathbb{Y}_\rho)  \overset{u}{\leq} {\asdim }(f_6^{-1}(e)).
\end{equation}
Thus fix $\rho > 0$, and for every $\lambda \in \Lambda$ pick a $\xi = \xi(\lambda) \in f_1^{-1}(\lambda)\subset\Xi$. 
Also note that \[(\lambda B'(e,\rho))\cap\Lambda \subset \lambda \left(B'(e,\rho)\cap\Lambda^2\right),\] since if $z=\lambda b=\widetilde\lambda$, for some $b\in B'(e,\rho)$, $\widetilde\lambda \in\Lambda$, then $b=\lambda^{-1}\widetilde\lambda\in \Lambda^2$. Now we have
\begin{eqnarray*} 
f_1^{-1}(B'(\lambda, \rho) \cap \Lambda) &=& f_1^{-1}\left(\left(f_1(\xi)B'(e, \rho) \right)\cap \Lambda\right) \quad \subset \quad f_3^{-1}\left(f_1(\xi)\left(B'(e, \rho) \cap \Lambda^2\right)\right)\\
&=& \{z \in \Xi^3 \mid f_3(z) \in  f_1(\xi)(B'(e, \rho) \cap \Lambda^2)\}\\
&\subset& \{z \in \Xi^3 \mid f_4(\xi^{-1}z) \in B'(e, \rho) \cap \Lambda^2\}\\
&\subset& \xi f_4^{-1}(B'(e, \rho) \cap \Lambda^2).
\end{eqnarray*}
Since left-multiplication by $\xi$ yields an isometry of $\Xi^\infty$, we have reduced \eqref{HurewiczToShow1} to proving that
\begin{equation}\label{HurewiczToShow2}
{\asdim}( f_4^{-1}(B'(e, \rho) \cap \Lambda^2)) \leq {\asdim }(f_6^{-1}(e))
\end{equation}
for every $\rho > 0$. Now by properness of $d'$ the ball $B'(e, \rho)$ is finite, and hence also its intersection with $\Lambda^2$ is finite, say $B'(e, \rho) \cap \Lambda^2 = \{\lambda_1, \dots, \lambda_N\}$. Since $f_2$ surjects onto $\Lambda^2$, we can find $\xi_1, \dots, \xi_N \in \Xi^2$ such that $f_2(\xi_i) = \lambda_i$. Define
\[
R \coloneqq  \max_{1 \leq i \leq N} d(e_\Xi, \xi_i), 
\] so also $\displaystyle R =  \max_{1 \leq i \leq N} d(e_\Xi, \xi_i^{-1})$.
Then for $i= 1, \dots, N$ and $\eta_i \in f_4^{-1}(\lambda_i)$, we have $\eta_i\xi_i^{-1} \in \Xi^4\Xi^2 = \Xi^6$, and $f_6(\eta_i\xi_i^{-1}) = \lambda_i\lambda_i^{-1} = e$. Since $d(\eta_i, f_6^{-1}(e))\leq d(\eta_i,\eta_i\xi_i^{-1}) = d(e_\Xi, \xi_i^{-1}) \leq R$, we deduce that
\[
f_4^{-1}(\lambda_i) \subset N_R(f_6^{-1}(e)),
\]
and consequently
\[
 f_4^{-1}(B'(e, \rho) \cap \Lambda^2) = \bigcup_{i=1}^N f_4^{-1}(\lambda_i)  \subset N_R(f_6^{-1}(e)).
\]
Therefore
\[
{\asdim}( f_4^{-1}(B'(e, \rho) \cap \Lambda^2) ) \leq {\asdim}(N_R(f_6^{-1}(e))) = {\asdim}(f_6^{-1}(e)),
\]
which establishes \eqref{HurewiczToShow2} and finishes the proof.
\end{proof}
\begin{remark}
According to Remark \ref{CoarseKernel},  the coarse kernel of a  global morphism $f: (\Xi, \Xi^\infty) \to (\Lambda, \Lambda^\infty)$ satisfies
\[
[\ker(f)]_c =  [\ker_2(f)]_c =  [\ker_3(f)]_c =  [\ker_4(f)]_c =  [\ker_5(f)]_c =  [\ker_6(f)]_c = \dots,
\]
so Theorem \ref{CorHurewiczConvenient} it now an immediate consequence of Theorem \ref{Thm: local asdim Hurewicz}.
\end{remark}

\section[Asdim of non-finitely generated countable approximate groups]{Asymptotic dimension of non-finitely generated countable approximate groups}

Let $(\Lambda, \Lambda^\infty)$ be a countable approximate group, so $\Lambda^\infty$ is a countable group.
If $\Lambda^\infty$ is finitely generated\footnote{In \cite{CHT}, when $\Lambda^\infty$ is finitely generated, then $(\Lambda, \Lambda^\infty)$ is said to be an \emph{algebraically finitely generated approximate group}.}, we can choose a finite symmetric generating set $S$ of $\Lambda^\infty$, which gives us the word metric $d_S$ on $\Lambda^\infty$, and $\asdim \Lambda$ can be computed as the asymptotic dimension of $\Lambda$ with respect to the restriction of the word metric $d_S$ (see Definition \ref{def: asdim-L}). If $\Lambda^\infty$ is \emph{not} finitely generated, then, as described in Remark \ref{rem: weight}, we choose a weighted word metric $d$ on $\Lambda^\infty$ and then compute the asymptotic dimension of $\Lambda$ with respect to the restriction of $d$, but working with such weighted word metrics is rather inconvenient. 

In the case of non-finitely-generated countable groups, we saw  in Theorem \ref{Thm: DS sup}
an alternative approach by Dranishnikov and Smith, in which they show that the asymptotic dimension of a countable group $\Gamma$ is the supremum over the asymptotic dimensions of all finitely generated subgroups of $\Gamma$. As we stated in the Introduction, using their techniques, in this section we will prove a somewhat analogous result for approximate groups, which we restate here for convenience.

\begin{usethmcounterof}{Thm: asdim FG} Let $(\Lambda, \Lambda^\infty)$ be a countable approximate group. Then 
\begin{equation}\label{eqn: asdim FG}
\asdim\Lambda = {\sup} \left(\underset{\Gamma \text{ finitely generated}}{\bigcup_{\Gamma \leq \Lambda^\infty}}\hspace{-6mm}\{\asdim\Xi\mid \Xi \subset \Lambda^2,\ \Xi \text{ is an approx.\ subgroup of } \Gamma\}\right).
\end{equation}
\end{usethmcounterof}

\begin{proof} If $(\Lambda, \Lambda^\infty)$ is an approximate group, then so is $(\Lambda^2, \Lambda^\infty)$, and by Corollary \ref{CorMonotonicity} (ii) we have $\asdim \Lambda =\asdim \Lambda^2$.
If $(\Xi, \Xi^\infty)$ is any approximate subgroup of $(\Lambda^2, \Lambda^\infty)$, then (iii) of Corollary \ref{CorMonotonicity} yields
\[
\asdim\Lambda=\asdim \Lambda^2   \geq \asdim\Xi.
\]
Passing to the supremum over all approximate subgroups $(\Xi, \Xi^\infty)$ of finitely ge\-ne\-ra\-ted subgroups of $\Lambda^\infty$ yields the inequality $\geq$ for the expression \eqref{eqn: asdim FG}.

Concerning the converse inequality, let us fix a weight function $w: \Lambda^\infty \to [0, \infty)$ and denote by $d$ the associated weighted word metric on $\Lambda^\infty$, as described in Remark \ref{rem: weight}. Recall that $d$ is left-invariant and proper. If the supremum on the right side of \eqref{eqn: asdim FG} equals $\infty$, then by the first part of the proof we have 
$\asdim \Lambda =\infty$ and the proof is finished. So let us assume that the supremum on the right in \eqref{eqn: asdim FG} equals $m$, for some $m\in\N_0$. We want to show that $\asdim\Lambda \leq m$, i.e., that for any $R>0$ (large enough), we can find a uniformly bounded, $(m+1)$-colored cover $\mathcal V=\bigcup_{i=0}^m \mathcal V^{(i)}$ of $\Lambda$, where each $\mathcal V^{(i)}$ is $R$-disjoint.

Let $R>0$ be given. We define
\[
T_R\coloneqq  \{g \in \Lambda^\infty \mid w(g) \leq R\} \quad \text{and} \quad \Gamma_R \coloneqq  \langle T_R \rangle.
\]
Since $w$ is a proper function, note that the set $T_R$ is finite, and thus the group $\Gamma_R$ is finitely generated. We also define 
\[
\Xi_R \coloneqq  \Lambda^2  \cap \Gamma_R. 
\]

Since both $\Lambda$ and $\Gamma_R$ are symmetric and contain $e$, the same is true of $\Xi_R$. In addition, we claim that for all sufficiently large $R$, the subset $\Xi_R \subset \Gamma_R$ is an approximate subgroup of the finitely-generated group $\Gamma_R$. Indeed, we know by Lemma \ref{Lemma: sym F} that there is a finite set $F_\Lambda \subset \Lambda^3$ such that $\Lambda^2\subset F_\Lambda \Lambda$, so by Corollary \ref{cor: commensur}, $\Lambda^4\subset F_\Lambda^2 \Lambda^2$. Let us mark by $F:=F_\Lambda^2$.
Then $F \subset \Lambda^\infty$ is a finite set such that $\Lambda^4 \subset F\Lambda^2$. Then $w$ is bounded on $F$ and thus $F \subset T_R$ for all sufficiently large $R$. For all such $R$ and $\xi_1, \xi_2 \in \Xi_R$ we then find $\lambda \in \Lambda^2$ and $f \in F$ such that $\xi_1\xi_2 = f\lambda$. Since $F \subset T_R$, we have 
\[\lambda = f^{-1}\xi_1\xi_2 \in T_R\Gamma_R\Gamma_R = \Gamma_R,\]
and thus $\lambda \in \Xi_R$. This shows that  $\Xi_R^2 \subset F\Xi_R$, and the claim follows. 
So let us agree that the constant $R>0$ given above was indeed large enough so that $\Xi_R$ defined above is an approximate subgroup of $\Gamma_R$. By our assumption, we have $\asdim \Xi_R \leq m$.

Now we shall focus on (left) cosets of $\Gamma_R$ in $\Lambda^\infty$. First, observe that if $g, h\in \Lambda^\infty$ are in two distinct $\Gamma_R$-cosets of $\Lambda^\infty$, then 
\begin{equation}\label{DifferentCosetR}d(g,h) > R.\end{equation} Namely, in this case the element $g^{-1}h$ is not contained in $\Gamma_R$, hence cannot be written as a product of elements of weight $\leq R$. Therefore $d(g,h) = \|g^{-1}h\|_w > R$, as claimed.

Furthermore, let us now choose a set $Z_R$ of representatives of $\Gamma_R$-cosets in $\Lambda^\infty$ so that
\[
\Lambda^\infty = \bigsqcup_{z \in Z_R} z\Gamma_R.
\]
Define $Z_R^0 \coloneqq  \{z \in Z_R \mid \Lambda \cap (z\Gamma_R) \neq \emptyset\}$ so that
\[
\Lambda = \bigsqcup_{z \in Z_R^0} \Lambda \cap (z\Gamma_R).
\]
We may assume without loss of generality that $Z_R^0 \subset \Lambda$, since every coset which intersects $\Lambda$ non-trivially admits a coset representative in $\Lambda$. Under this assumption we then have
\[
\Lambda = \bigsqcup_{z \in Z_R^0} \Lambda \cap (z\Gamma_R) = \bigsqcup_{z \in Z_R^0} z(z^{-1}\Lambda \cap \Gamma_R) \subset \bigcup_{z \in Z_R^0} z(\Lambda^2 \cap \Gamma_R) =  \bigcup_{z \in Z_R^0} z\Xi_R =: \widetilde{\Lambda}_R.
\]
Also note that 
\[
\Lambda \subset \widetilde{\Lambda}_R \subset  \bigcup_{z \in Z_R^0} z\Lambda^2 \subset \Lambda^3.
\]
Since $Z_R^0 \subset Z_R$, the elements of $Z_R^0$ represent distinct cosets of $\Gamma_R$. Consequently, $z\Gamma_R \cap z'\Gamma_R = \emptyset$ for all $z \neq z'$ in $Z_R^0$, and thus the union defining $\widetilde{\Lambda}_R$ is disjoint, i.e.,
\[
 \widetilde{\Lambda}_R =  \bigsqcup_{z \in Z_R^0} z\Xi_R.
\]
Moreover, if $z \neq z'$ are distinct elements of $Z_R^0$ and $g \in z\Xi_R$ and $h \in z'\Xi_R$, then $d(g,h)>R$ by \eqref{DifferentCosetR}. 

Since we know that $\asdim \Xi_R \leq m$, for the same $R$ given above we can choose the families $\mathcal U^{(0)}, \dots, \mathcal U^{(m)}$ of subsets of $\Xi_R$ so that each family $\mathcal U^{(i)}$ is $R$-disjoint, $D$-bounded (for some $D>0$), and so that $\mathcal U =\bigcup_{i=0}^m \mathcal{U}^{(i)}$ is a cover for $\Xi_R$. Using this $\mathcal U$, we define
\[
\mathcal V^{(i)}_R \coloneqq  \{zU \mid z\in Z_R^0, U \in \mathcal U^{(i)}\}, \quad i=0, \dots, m,
\]
that is, we transport the $(m+1)$-colored cover $\mathcal U$ of $\Xi_R$ to each $z\Xi_R$, and $\mathcal{V}^{(i)}_R$ represents the family of sets of the same color in $\bigsqcup_{z\in Z_R^0} z\Xi_R$.
Then the sets in $\mathcal V^{(i)}_R$ are still $D$-bounded, by left-invariance of the metric. Moreover, let $V = zU$ and $V'= z' U'$ be two distinct elements of $\mathcal V^{(i)}_R$. If $z\neq z'$, then $V$ and $V'$ are of distance at least $R$ by \eqref{DifferentCosetR}. If $z = z'$, then $U\neq U'$, and hence  $V$ and $V'$ are of distance at least $R$ by  $R$-disjointedness of $\mathcal U^{(i)}$. So $\mathcal{V}_R\coloneqq \bigcup_{i=0}^m \mathcal V^{(i)}_R$ is a uniformly bounded, $(m+1)$-colored cover of $\widetilde\Lambda_R$, with each $\mathcal V^{(i)}_R$ $R$-disjoint. Thus 
$\mathcal V:= \Lambda \cap \mathcal V_R= \bigcup_{i=0}^m (\Lambda\cap \mathcal V^{(i)}_R)$ is a uniformly bounded,  $(m+1)$-colored cover of $\Lambda$, with  each $\mathcal V^{(i)}:=\Lambda \cap \mathcal V^{(i)}_R$ $R$-disjoint. Therefore $\asdim \Lambda \leq m$,
which finishes the proof.
\end{proof}

Note that the theorem holds trivially if $\Lambda^\infty$ is finitely gene\-ra\-ted, so the interest is in the case where $\Lambda^\infty$ is \emph{not} finitely generated.
\begin{remark}
 If $(\Xi, \Xi^\infty)$ is an approximate subgroup of $(\Lambda, \Lambda^\infty)$ such that $\Xi$ (and hence $\Xi^\infty$) is contained in a finitely generated subgroup $\Gamma$ of $\Lambda^\infty$, then $\asdim\Xi$ can be computed as follows: Let $S$ be a finite generating set of $\Gamma$ and let $d_S$ be the associated word metric. Then $d_S$ restricts to a left-invariant proper word metric on $\Xi^\infty$ and hence to a proper metric on $\Xi$, and we thus have
\[
\asdim\Xi = \asdim(\Xi, d_S|_{\Xi \times \Xi}).
\]
Thus Theorem \ref{Thm: asdim FG} allows us to compute $\asdim\Lambda$ using only word metrics on finitely generated subgroups of $\Lambda^\infty$, rather than a weighted word metric on $\Lambda^\infty$. 
\end{remark}


\end{document}